\documentclass[preprint,12pt]{elsarticle}

\usepackage{amssymb}

\usepackage{amsmath}

\journal{Signal Processing}

\usepackage{color,xcolor}

\usepackage{url}

\usepackage{amsthm}
\newtheorem{thm}{Theorem}
\newtheorem{lem}[thm]{Lemma}

\newtheorem{remk}{Remark}
\newtheorem{assum}{Assumption}

\usepackage{algorithmic}
\usepackage{algorithm}

\usepackage{booktabs}
\usepackage{multirow}

\usepackage{epstopdf}

\usepackage[caption=false,font=normalsize,labelfont=sf,textfont=sf]{subfig}

\begin{document}

\begin{frontmatter}

\title{Cubic NK-SVD: An Algorithm for Designing Parametric Dictionary in Frequency Estimation}

\author[1]{Xiaozhi Liu}
\ead{xzliu@buaa.edu.cn}

\author[1]{Yong Xia \corref{cor1}}
\ead{yxia@buaa.edu.cn}

\cortext[cor1]{Corresponding author}

\affiliation[1]{organization={School of Mathematical Sciences},
            addressline={Beihang University}, 
            city={Beijing},
            postcode={100191},
            country={China}}

\begin{abstract}

We propose a novel parametric dictionary learning algorithm for line spectral estimation, applicable in both single measurement vector (SMV) and multiple measurement vectors (MMV) scenarios. 
This algorithm, termed {cubic Newtonized K-SVD (NK-SVD)}, extends the traditional K-SVD method by incorporating cubic regularization into Newton refinements.
The proposed Gauss-Seidel scheme not only enhances the accuracy of frequency estimation over the continuum but also achieves better convergence by incorporating higher-order derivative information. 
A key contribution of this work is the rigorous convergence analysis of the proposed algorithm within the Block Coordinate Descent (BCD) framework. To the best of our knowledge, this is the first convergence analysis of BCD with a higher-order regularization scheme.
Moreover, the convergence framework we develop is generalizable, providing a foundation for designing alternating minimization algorithms with higher-order regularization techniques.
Extensive simulations demonstrate that cubic NK-SVD outperforms state-of-the-art methods in both SMV and MMV settings, particularly excelling in the challenging task of recovering closely-spaced frequencies.
The code for our method is available at https://github.com/xzliu-opt/Cubic-NK-SVD.

\end{abstract}

\begin{keyword}
Parametric dictionary learning\sep line spectral estimation\sep Cubic Newton method\sep block coordinate descent

\end{keyword}

\end{frontmatter}

\section{Introduction}

The problem of line spectral estimation (LSE) has {gained} significant attention over the years due to its relevance in numerous signal processing applications, such as direction-of-arrival (DOA) estimation \cite{l1-svd} and channel estimation in wireless communications \cite{ce_app}.

Classical subspace methods, including multiple signal classification (MUSIC) \cite{music} and estimation of parameters by rotational invariant techniques (ESPRIT) \cite{esprit}, address the LSE problem by exploiting the low-rank structure of the autocorrelation matrix \cite{spectral-stoica}.
Under the assumption of Gaussian noise and complete data, these methods perform well when a sufficient number of snapshots are available. However, their effectiveness diminishes in challenging scenarios, such as with limited snapshots, closely spaced frequencies, unknown model orders, incomplete data, or low signal-to-noise ratio (SNR).

Recently, compressed sensing (CS) techniques have been employed to tackle the LSE problem \cite{cs}.
Methods such as $\ell_1$-SVD \cite{l1-svd} rely on a finite discretized dictionary to represent potential DOA estimates.
However, the accuracy of sparse signal reconstruction depends on the precise knowledge of the sparsifying dictionary, which is often impractical.
Moreover, CS methods require discretizing the continuous parameter space into a finite grid set, leading to grid mismatch errors \cite{chi-mismatch} when true parameters do not align with the predefined grid points. These methods are known as on-grid approaches.

To mitigate grid mismatch, off-grid methods have been introduced \cite{sure-ir,nomp,mnomp,ogsbi,sbl-de,superfast-lse}.
For example, super-resolution iterative reweighted (SURE-IR) algorithm addresses LSE using the majorization-minimization method but is computationally expensive.
In contrast, the Newtonized Orthogonal Matching Pursuit (NOMP) algorithm \cite{nomp} offers a faster solution by incorporating Newton-based refinement into the Orthogonal Matching Pursuit (OMP) framework \cite{omp}, with further extensions to multi-snapshot settings in the multi-snapshot NOMP (MNOMP) algorithm \cite{mnomp}.
In \cite{ogsbi,sbl-de,superfast-lse}, a Bayesian view is taken on the LSE problem.
Besides these efforts, another line of research focuses on gridless techniques, which avoid discretization by utilizing continuous dictionaries.
{The work in \cite{anm} presents an atomic norm minimization (ANM) approach to solve the LSE problem using continuous dictionaries, and this approach is extended to the MMV setting in \cite{anm-mmv}. Furthermore, \cite{emac} introduces a gridless method based on enhanced matrix completion (EMaC).}
Despite their theoretical advantages, these methods involve solving semidefinite programming (SDP), which can be computationally demanding.

{The aforementioned off-grid methods} {\cite{sure-ir,nomp,mnomp,ogsbi,sbl-de,superfast-lse}} can be {viewed as} specialized parametric dictionary learning {(DL)} methods \cite{para-dl,para-dl-gd}, where each atom is {parameterized} by a single frequency.
These methods typically employ an alternating minimization strategy between the parametric atoms and their corresponding gains.
{Some approaches directly apply parametric DL methods to LSE \cite{dl-dicrefcs-em,dl-imat-1bit}, but these are generally first-order algorithms that do not leverage second-order information.}
Moreover, they lack convergence analysis and utilize a parallel update strategy for all parametric atoms, without leveraging a Gauss-Seidel-like accelerated strategy, as employed in the K-SVD algorithm \cite{k-svd}.
The authors in \cite{nomp} also emphasize that the decision-feedback mechanism in this {alternative strategy} {can significantly improve performance by} handling interference among the underlying atoms.

In this paper, we address the LSE problem within a parametric DL framework to mitigate the effect of grid mismatch.
Inspired by the classical DL approach \cite{k-svd}, we extend the K-SVD algorithm \cite{k-svd} to the frequency estimation problem by incorporating the cubic regularization of Newton’s method \cite{cubic_nesterov_06}.
The proposed approach, termed cubic Newtonized K-SVD (cubic NK-SVD) algorithm, is suitable for both SMV and MMV scenarios.
Another contribution of our work is the first rigorous establishment of the convergence analysis of Block Coordinate Descent (BCD) \cite{bcd_bolte_14} using higher-order derivative information. This analytical framework is not only applicable to the analysis of our proposed cubic NK-SVD algorithm but can also be used in the design of general alternating minimization algorithms with higher-order regularization.
{Comprehensive numerical results validate the theoretical findings and demonstrate the superior accuracy and computational efficiency of the proposed cubic NK-SVD algorithm.}

The following notation is used throughout this paper:
$\boldsymbol{A}$ denotes a matrix;
$\boldsymbol{a}$ represents a vector;
$a$ is a scalar;
$\boldsymbol{A}^T$ and $\boldsymbol{A}^H$ represent the transpose and conjugate transpose of $\boldsymbol{A}$, respectively;
$\|\boldsymbol{A}\|_F$ is the Frobenius norm of $\boldsymbol{A}$;
$\|\boldsymbol{a}\|_p$ is the $p$-norm of the vector $\boldsymbol{a}$;
$\|\boldsymbol{a}\|$ specifically denotes the 2-norm of $\boldsymbol{a}$;
{$\operatorname{Diag}(\boldsymbol{a})$} denotes a diagonal matrix with the elements of vector $\boldsymbol{a}$ on its diagonal;
$\boldsymbol{I}_N$ is the $N\times N$ identity matrix;
$\boldsymbol{0}_{M\times N}$ represents the $M \times N$ all-zeros matrix;
The real part of a variable is denoted by $\Re{\{\cdot\}}$.

The rest of the paper is organized as follows.
In Section \ref{sec:problem}, we formulate the LSE problem as a parametric DL problem.
The proposed cubic NK-SVD algorithm is developed in Section \ref{sec:algorithm}.
Section \ref{sec:convergence} offers a convergence analysis based on the BCD framework. 
Simulation results are provided in Section \ref{sec:simulation}.
Finally, the paper is concluded in Section \ref{sec:conclusion}.

\section{Problem Formulation} \label{sec:problem}

{The problem of LSE involves observing a signal composed of multiple complex sinusoids, expressed as:}
\begin{equation}
	\boldsymbol{y}(t) = \boldsymbol{A}(\boldsymbol{\theta}) \boldsymbol{s}(t) + \boldsymbol{e}(t),\quad t = 1,\cdots, T,
	\label{smv}
\end{equation}
where $\boldsymbol{y}(t) = \left[y_1(t),  \cdots  ,y_N(t)\right]^T$, $\boldsymbol{\theta} = \left[\theta_1,  \cdots , \theta_K \right]^T$,  $\boldsymbol{s}(t) = \left[s_1(t),  \cdots , s_K(t)\right]^T$, and $\boldsymbol{e}(t) = \left[e_1(t),  \cdots , e_N(t)\right]^T$. The matrix $\boldsymbol{A}(\boldsymbol{\theta}) =\left[\boldsymbol{a}(\theta_1),  \cdots  ,\boldsymbol{a}(\theta_K)\right]^T$ is known as the array manifold, where $\boldsymbol{a}(\theta_k)$ represents the steering vector for the $k$-th source. The steering vector for a frequency $\theta \in [0, 2\pi)$ is defined as:
\begin{equation}
	\boldsymbol{a}(\theta) =\dfrac{1}{\sqrt{N}}\left[1 , e^{j \theta} \cdots,  e^{j (N-1) \theta}\right]^T,
	\label{eq:dft_atom}
\end{equation}
the model \eqref{smv} can be rewritten in a matrix form as
$$
\boldsymbol{Y} = \boldsymbol{A}(\boldsymbol{\theta}) \boldsymbol{S} + \boldsymbol{E},
$$
where $\boldsymbol{Y} = [\boldsymbol{y}(1),\cdots ,\boldsymbol{y}(T)]$ is the observation matrix consisting of $T$ observed vectors, $\boldsymbol{S} = [\boldsymbol{s}(1),\cdots ,\boldsymbol{s}(T)]$, and $\boldsymbol{E}$ is the noise matrix.

In the context of CS, the problem is further complicated when a sensing matrix $\boldsymbol{\Phi}\in\mathbb{C}^{M\times N}$ (with $M\ll N$) is applied. Here, the atoms transform into $\tilde{\boldsymbol{a}}(\theta) = \boldsymbol{\Phi} \boldsymbol{a}(\theta)$, and the new manifold becomes $\tilde{\boldsymbol{A}}(\boldsymbol{\theta}) =\left[\tilde{\boldsymbol{a}}(\theta_1),  \cdots  ,\tilde{\boldsymbol{a}}(\theta_K)\right]$. The goal of this paper is to estimate the number of sources $K$, the unknown frequencies $\boldsymbol{\theta}$, and the gains $\boldsymbol{S}$ based on the observed data $\boldsymbol{Y}$ and the sensing matrix $\boldsymbol{\Phi}$.

To cast the LSE problem as a parametric DL problem \cite{k-svd}, we introduce an overcomplete dictionary $\boldsymbol{D}(\boldsymbol{\tilde{\theta}})$ that covers all potential source locations.
Let $\boldsymbol{\tilde{\theta}} = [\tilde{\theta}_1,\cdots,\tilde{\theta}_R]$ be a grid sampling possible source locations, where the number of grid points $R$ is typically much larger than both the number of sources $K$ and the number of sensors $N$.
The parametric dictionary is constructed with steering vectors corresponding to each potential source location as its columns: $\boldsymbol{D}(\boldsymbol{\tilde{\theta}})=[\boldsymbol{a}(\tilde{\theta}_1),\cdots,\boldsymbol{a}(\tilde{\theta}_R)]$.

We represent the signal field using the matrix $\boldsymbol{X} = [\boldsymbol{x}(1),\cdots,\boldsymbol{x}(T)]$. In the MMV case, the model becomes:
$$
\boldsymbol{Y} = \boldsymbol{D}(\boldsymbol{\tilde{\theta}}) \boldsymbol{X}+\boldsymbol{E}.
$$

This overcomplete representation enables the reformulation of the LSE problem into the task of designing the parametric dictionary $\boldsymbol{D}(\boldsymbol{\tilde{\theta}})$ to achieve optimal sparse representations of $\boldsymbol{X}$ under strict sparsity constraints.

\section{Cubic NK-SVD Algorithm} \label{sec:algorithm}

In this section, we present the cubic NK-SVD algorithm, which is designed to estimate the unknown parameters $\boldsymbol{\theta}$ and the sparse matrix $\boldsymbol{X}$ by constructing a parametric dictionary $\boldsymbol{D}(\boldsymbol{\tilde{\theta}})$. This approach adapts the classical K-SVD algorithm \cite{k-svd} by incorporating Cubic Newton refinements to address the potential basis mismatch that arises from discretizing a continuous parameter space. For the sake of exposition, we begin by deriving it in the context of the MMV. The SMV case is simply a special instance of this formulation, obtained by setting the number of measurement vectors $T = 1$. The objective is to find the best possible parametric dictionary $\boldsymbol{D}(\boldsymbol{\tilde{\theta}})$ that provides a sparse representation of the given training set $\boldsymbol{Y}$. This leads to the following optimization problem:
\begin{equation}
	\begin{aligned}
		\min_{\boldsymbol{\tilde{\theta}}, \boldsymbol{X}} & \ \|\boldsymbol{Y} - \boldsymbol{D}(\boldsymbol{\tilde{\theta}}) \boldsymbol{X} \|_F^2 \\
		\text { s.t. } & \ \ \|{{\boldsymbol{x}}_i}\|_0 \leq {K},\ \forall i = 1,\cdots, T,
	\end{aligned}
	\label{eq:dl-problem}
\end{equation}
where ${{\boldsymbol{x}}_i}$ is the $i$-th column of the matrix $\boldsymbol{X}$.

{In our algorithm, we iteratively minimize the objective in \eqref{eq:dl-problem} by alternating between the sparse coding and atom update stages. During the sparse coding stage, the parametric dictionary $\boldsymbol{D}(\boldsymbol{\tilde{\theta}})$ is fixed, and the focus is on determining the optimal sparse representation $\boldsymbol{X}$.}
There are various algorithms for solving the sparse representation problem, such as OMP \cite{omp} and fast iterative shrinkage-thresholding algorithm (FISTA) \cite{fista}.
{In our algorithm, we select OMP \cite{omp} for its balanced performance in terms of computational complexity and representational accuracy.}
We employ an error bound $\epsilon$, determined by an estimate of noise power, as the stopping criterion, {eliminating the need for prior knowledge of model order $K$ \cite{omp}. Specifically, OMP expands the support until the residual error falls below the threshold $\epsilon$. In the atom update stage, we adopt a Gauss–Seidel-like accelerated scheme, where the atoms in $\boldsymbol{D}(\boldsymbol{\tilde{\theta}})$ are updated sequentially, and the corresponding coefficients are adjusted. The key distinction from the K-SVD algorithm \cite{k-svd} lies in the parametric structure of the atoms $\boldsymbol{a}(\tilde{\theta}_k)$ within the dictionary\footnote{The discussion here is not limited to Fourier atom \eqref{eq:dft_atom}. The extension to general atomic sets is {straightforward}.}.}
Moreover, updates are only required for the parameters $\boldsymbol{\tilde{\theta}}$ within the dictionary $\boldsymbol{D}(\boldsymbol{\tilde{\theta}})$.
Assuming that both $\boldsymbol{D}(\boldsymbol{\tilde{\theta}})$ and $\boldsymbol{X}$ are fixed except $\theta_k$ and the corresponding coefficients $\boldsymbol{x}_T^k$, 
the subproblem of estimating $\theta_k$ and $\boldsymbol{x}_T^k$ can be transformed into minimizing the following objective
\begin{equation}
	\begin{aligned}
		S(\tilde{\theta}_k, \boldsymbol{x}_T^k)=& \|\boldsymbol{Y} - \boldsymbol{D}(\boldsymbol{\tilde{\theta}}) \boldsymbol{X} \|_F^2 \\
		=& \|\boldsymbol{Y} -\sum_{j=1}^{R} \boldsymbol{a}(\tilde{\theta}_j) \boldsymbol{x}_T^j \|_F^2 \\
		=& \|(\boldsymbol{Y} -\sum_{j\neq k} \boldsymbol{a}(\tilde{\theta}_j) \boldsymbol{x}_T^j) - \boldsymbol{a}(\tilde{\theta}_k) \boldsymbol{x}_T^k \|_F^2\\
		=&  \|\boldsymbol{E_k} - \boldsymbol{a}(\tilde{\theta}_k) \boldsymbol{x}_T^k \|_F^2,
	\end{aligned}
	\label{eq:f_norm}
\end{equation}
where $\boldsymbol{x}_T^j$ is the coefficients that correspond to the $j$-th atom $\boldsymbol{a}(\tilde{\theta}_j)$, i.e., the $j$-th row of the matrix $\boldsymbol{X}$. The matrix $\boldsymbol{E_k}$ stands for the error for all the $T$ examples when the $k$-th atom is removed.

{Directly optimizing $S(\tilde{\theta}_k, \boldsymbol{x}_T^k)$ over both $\tilde{\theta}_k$ and $\boldsymbol{x}_T^k$ is a challenging task. To address this, we implement a two-stage procedure:
	(1) Coarse estimation stage: In this step, we first restrict $\tilde{\theta}_k$ to a discrete set and obtain a coarse estimate by applying singular value decomposition (SVD) to approximate $\boldsymbol{E_k}$ with the closest rank-1 matrix in the Frobenius norm \cite{k-svd};
	(2) {Cubic Newton refinement stage}: Following the coarse estimation, we iteratively refine both $\tilde{\theta}_k$ and $\boldsymbol{x}_T^k$ using the cubic regularization of Newton’s method \cite{cubic_nesterov_06}, which improves the accuracy of the parameter estimates.}

{To enforce the sparsity pattern in $\boldsymbol{x}_T^{k}$, we employ a scheme similar to that outlined in \cite{k-svd}. Let $\mathcal{I}_k$ denote the set of indices corresponding to the non-zero entries of $\boldsymbol{x}_T^{k}$,}
i.e., $\mathcal{I}_k=\{\ell\mid \boldsymbol{x}_T^k(\ell)\neq 0\}$.
Define $\boldsymbol{E_k^R} = (\boldsymbol{E_k})_{\mathcal{I}_k}$, where the subscript $\mathcal{I}_k$ indicates the restriction of the matrix to the columns {indexed by} $\mathcal{I}_k$.
Similarly, $\boldsymbol{x}_R^k = (\boldsymbol{x}_T^k)_{\mathcal{I}_k}$ is derived by removing the zero entries from $\boldsymbol{x}_T^k$.
{With this notation, minimizing \eqref{eq:f_norm} can be performed directly using SVD, without considering the parametric structure of the atoms.}
Applying SVD to the restricted matrix $\boldsymbol{E_k^R}$, we decompose it as $\boldsymbol{E_k^R} = \boldsymbol{U} \boldsymbol{\Sigma} \boldsymbol{V}^H$.
Let $\boldsymbol{u}_1$ represent the first column of $\boldsymbol{U}$.
Given the structure of the atom, the atom corresponding to the optimal $\tilde{\theta}_k$ should be close to $\boldsymbol{u}_1$.
Using this insight, we can derive a coarse estimate of $\tilde{\theta}_k$ in the coarse estimation stage.

\textit{Coarse estimation stage}: We obtain a coarse estimate of $\tilde{\theta}_k$ by restricting it to a finite discrete set denoted by ${\boldsymbol{\Omega}} = \{k(2\pi/{\gamma} N):k=0,1,\cdots,({\gamma} N-1)\}$, where {$\gamma$} represents the oversampling factor relative to the Discrete Fourier Transform (DFT) grid. The output $(\tilde{\theta}_k)_c$ is determined by solving the following optimization problem:
\begin{equation}
	(\tilde{\theta}_k)_c = \arg \max_{\theta \in {\boldsymbol{\Omega}}} |\boldsymbol{a}(\theta)^H\boldsymbol{u}_1|.
	\label{eq:omega_coarse}
\end{equation}
Once $\tilde{\theta}_k$ is fixed, the optimal coefficients that minimize $S(\tilde{\theta}_k, \boldsymbol{x}_T^k)$ are given by:
\begin{equation}
	\hat{\boldsymbol{x}}_R^k = \boldsymbol{a}(\tilde{\theta}_k)^H \boldsymbol{E_k^{R}},
	\label{eq:gain_coarse}
\end{equation}
where $\hat{\boldsymbol{x}}_R^k = (\hat{\boldsymbol{x}}_T^k)_{\mathcal{I}_k}$.
Consequently, the corresponding coefficients in index set $\mathcal{I}_k$ at this stage are represented as $\boldsymbol{a}((\tilde{\theta}_k)_c)^H \boldsymbol{E_k^R}$.

\textit{Cubic Newton refinement stage}: {To mitigate the effects of grid mismatch encountered during the coarse estimation stage, we employ a refinement procedure based on Cubic Newton method, allowing for continuous parameter estimation.}
Let $(\hat{\theta}_k, \hat{\boldsymbol{x}}_T^k)$ denote the current estimate. 
Note that the {partial Hessian} of our objective function {$S$ with respect to $\theta$} is {Lipschitz continuous}, as defined in Assumption \ref{assum:1}(ii).
Thus, we can try to {refine the frequency $\hat{\theta}_k$ to} {$\hat{\theta}_k^{\prime}$} using cubic regularization of Newton’s method proposed in \cite{cubic_nesterov_06}.
The refinement step for frequency is given by
\begin{equation}
	\hat{\theta}_k^{\prime} \in \arg \min_{\theta} \xi_{\hat{\theta}_k, \hat{\boldsymbol{x}}^k_T}^{L(\hat{\boldsymbol{x}}^k_T)}(\theta), 
	\label{eq:newton_refine}
\end{equation}
where the auxiliary function
$$
\begin{aligned}
	\xi_{\hat{\theta}_k, \hat{\boldsymbol{x}}^k_T}^{L(\hat{\boldsymbol{x}}^k_T)}(\theta) =&  \left\langle {{\nabla_{\theta}}S{\left(\hat{\theta}_k, \hat{\boldsymbol{x}}^k_T\right)}}, \theta-\theta_k\right\rangle+\frac{1}{2}\left\langle {{\nabla^2_{\theta\theta}}S\left(\hat{\theta}_k, \hat{\boldsymbol{x}}^k_T\right)}( \theta-\theta_k),  \theta-\theta_k\right\rangle \\
	&+\frac{{L(\hat{\boldsymbol{x}}^k_T)}}{6}|\theta-\theta_k|^3
\end{aligned}
$$
is an upper second-order approximation for our objective function.
{$L(\hat{\boldsymbol{x}}^k_T)$ is the {Lipschitz continuous moduli} of objective function given $\hat{\boldsymbol{x}}^k_T$.}
{The gain parameter is then updated according to \eqref{eq:gain_coarse}.}
In Section \ref{sec:convergence}, we will provide a detailed analysis of the aforementioned strategy and demonstrate the convergence guarantee of this scheme throught a BCD framework \cite{bcd_bolte_14}.

\begin{algorithm}[t] 
	\caption{{Cubic Newtonized K-SVD (Cubic NK-SVD)}} 		\label{alg:nk-svd} 
	\begin{algorithmic}[1]
		\STATE {\bfseries Input:}{ $\boldsymbol{Y}$, $R$, $\gamma$, $\epsilon$.} 
		\STATE {\bfseries Initialize:} $\widehat{K}=R$, $\boldsymbol{f}=\{k(2\pi/R):k=0,1,\cdots,R-1\}$ and  $\boldsymbol{\Omega} = \{k(2\pi/{\gamma} N):k=0,1,\cdots,({\gamma} N-1)\}$.
		\STATE {\bfseries Output:} {Updated $\widehat{K}$, $\boldsymbol{f}$, and $\boldsymbol{X}$}.
		\REPEAT 
		\STATE {\bfseries \textit{Sparse Coding Stage}}:
		\STATE $\boldsymbol{X}$ is updated using OMP {with an error bound $\epsilon$}.
		\STATE {\bfseries \textit{Atom Update Stage}}:
		\FOR{$k=$ 1 \textbf{to} $R$}
		\STATE Compute $\mathcal{I}_k=\{\ell\mid {\boldsymbol{x}}_T^k(\ell)\neq 0\}$.
		\IF{\rm $\mathcal{I}_k=\emptyset$}
		\STATE {Eliminate the unused atom $\boldsymbol{a}(\tilde{\theta}_j)$}.
		\STATE $\widehat{K} = \widehat{K} - 1$.
		\ELSE
		\STATE Compute error $\boldsymbol{E_k} = \boldsymbol{Y} -\sum_{j\neq k} \boldsymbol{a}(\tilde{\theta}_j) {\boldsymbol{x}}_T^j$.
		\STATE {Retrict $\boldsymbol{E_k}$ by $\boldsymbol{E_k^R} = (\boldsymbol{E_k})_{\mathcal{I}_k}$}.
		\STATE Compute first left singular vector $\boldsymbol{u}_1$ of {$\boldsymbol{E_k^R}$}.
		\STATE \textit{Coarse Estimation}: Update $(\tilde{\theta}_k)_c$ within $\boldsymbol{\Omega}$ by \eqref{eq:omega_coarse} and its corresponding gain $({\boldsymbol{x}}_T^k)_c$ by \eqref{eq:gain_coarse}.
		\STATE \textit{Cubic Newton Refinement}: Refine $(\hat{\theta}_k, \hat{\boldsymbol{x}}_T^k)$ using \eqref{eq:newton_refine}.
		\ENDIF
		\ENDFOR
		\UNTIL{stopping criterion met}
	\end{algorithmic}
\end{algorithm}

{A common occurrence in the algorithm is the non-utilization of certain atoms. Specifically, if the atom $\boldsymbol{a}(\tilde{\theta}_j)$ is unused, it implies that its corresponding support set $\mathcal{I}_j$ is empty. As a result, updating such an atom becomes unnecessary, as it cannot be modified. To address this, we implement a strategy of removing these unused atoms, thereby retaining only the relevant ones.}
This strategy not only yields a more efficient set of active atoms but also reduces the computational complexity of the algorithm, as the number of useful atoms is considerably smaller than the initial grid resolution $R$. 
The {cubic NK-SVD} algorithm\footnote{{It can be easily extended to compressive scenarios with a compressive manifold $\boldsymbol{\tilde{D}}(\boldsymbol{\tilde{\theta}})= \boldsymbol{\Phi} \boldsymbol{D}(\boldsymbol{\tilde{\theta}})$.}} is summarized in Algorithm \ref{alg:nk-svd}.

\begin{remk} 
	Stopping criterion: {Cubic NK-SVD} algorithm will be terminated if the difference in function values between two consecutive steps is less than {$10^{-3}$}, or a maximum of {30} iterations is reached.
	
\end{remk}

\section{{Convergence}} \label{sec:convergence}

We begin by addressing the estimation of a single frequency, subsequently extending the methodology to handle a mixture of frequencies. Our focus initially centers on solving the nonconvex minimization problem involving two blocks of variables, $(x,\boldsymbol{y}) \in \mathbb{R}^1\times \mathbb{C}^T$, which is formulated as:
$$
\min_{{x}, \boldsymbol{y}}  H({x}, \boldsymbol{y}).
$$
In line with the approach detailed in \cite{cubic_nesterov_06} and \cite{bcd_bolte_14}, we impose the following blanket assumption.
\begin{assum} \label{assum:1}
	(i) Assume that the objective function $H(x,\boldsymbol{y})$ is bounded below, i.e.,
	$\inf_{\mathbb{R}^1\times \mathbb{C}^T} H \geq H^*.$
	
	(ii) For any fixed $\boldsymbol{y}$, the function $x \rightarrow H(x,\boldsymbol{y})$ is $C_{L(\boldsymbol{y})}^{2,2}$, meaning that the {partial Hessian} $\nabla^2_{xx} H(x,\boldsymbol{y})$ is Lipschitz continuous with moduli $L(\boldsymbol{y})$. Specifically, we have:
	$$
	\left\|\nabla^2_{xx} H\left(x_1, \boldsymbol{y}\right)-\nabla^2_{xx} H\left(x_2, \boldsymbol{\boldsymbol{y}}\right)\right\| \leq L(\boldsymbol{y})\left\|x_1-x_2\right\|, \forall x_1, x_2 \in \mathbb{R}^1.
	$$
	
	(iii) For any fixed $x$ the function $\boldsymbol{y} \rightarrow H(x,\boldsymbol{y})$ is assumed to be {strongly convex (or more specifically, $\nu$-strongly convex)} in $\mathbb{C}^{T}$.
	In other words, there exists $\nu>0$ such that
	$$
	{(\nabla^2_{\boldsymbol{y}\boldsymbol{y}})_c H(x,\boldsymbol{y})} \succeq \frac{\nu}{2} {\boldsymbol{I}_{2T}} \succ {\boldsymbol{0}},
	$$
	where we use the notation ${(\nabla^2_{\boldsymbol{y}\boldsymbol{y}})_c H(\boldsymbol{y})}$ to refer to the complex form of the partial Hessian matrix of $H(x,\boldsymbol{y})$ \cite{adaptation}.
	
	(iv) There exists $\lambda^-, \lambda^+$ such that
	\begin{equation}
		\inf \left\{L\left({\boldsymbol{y}}^k\right): k \in \mathbb{N}\right\} \geq \lambda^{-},
		\label{eq:lip_lower_bound}
	\end{equation}
	\begin{equation}
		\sup \left\{L\left({\boldsymbol{y}}^k\right): k \in \mathbb{N}\right\} \leq \lambda^{+}.
		\label{eq:lip_upper_bound}
	\end{equation}
\end{assum}
A few words on Assumption \ref{assum:1} are now in order.
\begin{remk}
	(i) The partial Lipschitz properties required in Assumption \ref{assum:1}(ii) are at the heart of
	{the cubic NK-SVD algorithm} which is designed to fully exploit the Hessian-Lipschitz property of the problem at hand.
	
	(ii)
	In our problem, the goal is to estimate unknown frequencies and corresponding gains.
	To be specific, the objective function $H(x,\boldsymbol{y})$ is 
	\begin{equation}
		H(x,{\boldsymbol{y}}) = \|\boldsymbol{E} - \boldsymbol{a}(x){\boldsymbol{y}^H}\|_F^2,
		\label{eq:H_doa}
	\end{equation} 
	where $\boldsymbol{a}(x)$ is the steering vector of frequency $x\in \mathbb{R}^1$.
	It is easy to verify that $H(x,{\boldsymbol{y}})$ in \eqref{eq:H_doa} satisfies the above Assumption \ref{assum:1} (See Appendix \ref{appendix:property_H}).
\end{remk}

As outlined in Algorithm \ref{alg:nk-svd},
{the cubic NK-SVD algorithm} can be viewed as a Gauss-Seidel iteration scheme, popularly known as alternating minimization.
Starting with an initial point $(x^0,{\boldsymbol{y}}^0)$, we generate a sequence $\{(x^k,{\boldsymbol{y}}^k)\}_{k\in \mathbb{N}}$ via the scheme:
\begin{equation}
	x^{k+1} \in \arg \min_x \xi_{x^k, {\boldsymbol{y}}^k}^{{L({\boldsymbol{y}}^k)}}(x), \label{eq:x_subproblem}
\end{equation}
\begin{equation}
	{\boldsymbol{y}}^{k+1} \in \arg \min_{\boldsymbol{y}} H \left(x^{k+1}, \boldsymbol{y}\right), \label{eq:y_subproblem}
\end{equation}
where the auxiliary function is defined as:
$$
\begin{aligned}
	\xi_{x^k, {\boldsymbol{y}}^k}^{{L(\boldsymbol{y}^k)}}(x) =&  \left\langle \nabla_{x}H\left(x^k, {\boldsymbol{y}}^k\right), x-x^{k}\right\rangle+\frac{1}{2}\left\langle \nabla^2_{xx}H\left(x^k, {\boldsymbol{y}}^k\right)(x-x^{k}), x-x^{k}\right\rangle \\
	&+\frac{{L(\boldsymbol{y}^k)}}{6}|x-x^{k}|^3.
\end{aligned}
$$
This function serves as an upper second-order approximation for our objective function under Assumption \ref{assum:1}(ii):
$$
H(x,{\boldsymbol{y}}^k) \leq \xi_{x^k, {\boldsymbol{y}}^k}^{{L(\boldsymbol{y}^k)}}(x), \quad \forall x \in \mathbb{R}^1.
$$
This is exactly the approach we analyze in this paper. This scheme is called cubic regularization of Newton’s method \cite{cubic_nesterov_06}.
Moreover, solving the subproblem \eqref{eq:x_subproblem} is not challenging, as it is a single-variable optimization problem. This problem can be reformulated as solving the following nonlinear equation:
$$
\nabla_{x}H\left(x^k, {\boldsymbol{y}}^k\right)+\nabla^2_{xx}H\left(x^k, {\boldsymbol{y}}^k\right)(x-x^k)+\frac{1}{2} {L(\boldsymbol{y}^k)}|x-x^k| \cdot(x-x^k)=0.
$$
Due to the strong convexity assumption of $h$ with respect to $\boldsymbol{y}$, 
the minimum of \eqref{eq:y_subproblem} in each step is uniquely attained:
$$
\boldsymbol{y}^{k+1} = \boldsymbol{E}^H\boldsymbol{a}(x^{k+1}).
$$
Convergence results for the Gauss-Seidel method, also known as the BCD approach, can be found in several studies \cite{tseng2001convergence,bcd_bolte_14}.
In the light of general methodology in \cite{bcd_bolte_14}, we characterize the convergence results {for the cubic NK-SVD algorithm}.

We first establish some useful properties for {the cubic NK-SVD algorithm} under our Assumption \ref{assum:1}.
For convenience, we will use the notation
$$
\boldsymbol{z}^k := (x^k,\boldsymbol{y}^k), \quad \forall k \geq 0.
$$

\begin{lem} \label{lem:convergence}
	Suppose that Assumption \ref{assum:1} hold. Let $\{\boldsymbol{z}^k\}_{k\in\mathbb{N}}$ be a sequence iteratively generated by \eqref{eq:x_subproblem} and \eqref{eq:y_subproblem}. The following assertions hold.
	
	(i) The sequence $\{H(\boldsymbol{z}^k)\}_{k\in\mathbb{N}}$ is nonincreasing, and in particular, for any $k\geq 0$, we have
	$$
	\frac{\lambda^-}{12}\left|x^{k+1}-x^k\right|^3 + \left\|\boldsymbol{y}^{k+1}-\boldsymbol{y}^k\right\|^2 \leq H\left(\boldsymbol{z}^k\right)-H\left(\boldsymbol{z}^{k+1}\right).
	$$
	
	(ii) We have
	\begin{equation}
		\sum_{k=0}^{\infty}\left|x^{k+1}-x^k\right|^{{3}}+\left\|\boldsymbol{y}^{k+1}-\boldsymbol{y}^k\right\|_3^{{3}}=\sum_{k=0}^{\infty}\left\|\boldsymbol{z}^{k+1}-\boldsymbol{z}^k\right\|_3^{{3}}<\infty,
		\label{eq:lem_point_conver}
	\end{equation}
	and hence $\lim_{k \rightarrow \infty}\left\|\boldsymbol{z}^{k+1}-\boldsymbol{z}^k\right\|=0$.
\end{lem}

\begin{proof}
	See Appendix \ref{appendix:proof_lem_convergence}.
\end{proof}

{To demonstrate the effectiveness of our approach, we now illustrate how iterative schemes in \eqref{eq:x_subproblem} and \eqref{eq:y_subproblem}, which are integral components of {the cubic NK-SVD algorithm}, can be applied to efficiently solve the LSE problem, with the objective function $H(x,\boldsymbol{y})$ defined in \eqref{eq:H_doa}.}

\begin{lem} \label{lem:gradient_gap}
	Suppose that Assumptions \ref{assum:1} hold.
	Let $\{\boldsymbol{z}^k\}_{k\in\mathbb{N}}$ be a sequence generated by \eqref{eq:x_subproblem} and \eqref{eq:y_subproblem}
	{which is assumed to be bounded, }{i.e., there exists {$\sigma$} such that $\|\boldsymbol{z}^k\| \leq \sigma,\forall k \in \mathbb{N}$}.
	For each positive integer $k$, {we have}
	\begin{equation}
		\|\nabla_x H(x^{k+1},\boldsymbol{y}^{k+1}) \| \leq \lambda^+ \|x^{k+1} - x^k \|^2 + 2{\sigma} \|\boldsymbol{y}^{k+1} - \boldsymbol{y}^k\|,
		\label{eq:lem_grad_x}
	\end{equation}
	and {Wirtinger derivative}
	$$
	\|\nabla_{\boldsymbol{y}} H(x^{k+1},\boldsymbol{y}^{k+1}) \| = 0.
	$$
	Then, 
	$$
	\|\nabla_{\boldsymbol{z}} H(\boldsymbol{z}^{k+1}) \| \leq \|\nabla_x H(x^{k+1},\boldsymbol{y}^{k+1}) \| + \|\nabla_{\boldsymbol{y}} H(x^{k+1},\boldsymbol{y}^{k+1}) \| \leq {\rho_2} \|\boldsymbol{z}^{k+1} - \boldsymbol{z}^k\|,
	$$
	where $\rho_2 = \max \{\lambda^+, 2\sigma\}$.
\end{lem}

\begin{proof}
	See Appendix \ref{appendix:proof_gradient_gap}.
\end{proof}

We now summarize {several properties} of the limit point set for sequences generated by the cubic NK-SVD algorithm.
Let $\{\boldsymbol{z}^k\}_{k\in\mathbb{N}}$ denote the sequence produced by the algorithm from a initial point $\boldsymbol{z}^0$.
The set of all limit points is denoted by $\omega (\boldsymbol{z}^0)$, defined as follows:
$$
\begin{aligned}
	\omega (\boldsymbol{z}^0) = \{\bar{\boldsymbol{z}}\in \mathbb{R}^1\times \mathbb{C}^T: \exists\ \text{an increasing sequence of integers} \\ \{k_l\}_{l\in\mathbb{N}},\ \text{sucn that}\ \boldsymbol{z}^{k_l} \rightarrow\bar{\boldsymbol{z}}\ as\ l \rightarrow \infty \}.
\end{aligned}
$$

\begin{thm} \label{thm:critical}
	Suppose that Assumption \ref{assum:1} hold. Let $\{\boldsymbol{z}^k\}_{k\in\mathbb{N}}$ be a sequence generated from {the cubic NK-SVD algorithm} which is assumed to be bounded.
	The following assertions hold.
	
	(i) $\emptyset\neq \omega(\boldsymbol{z}^0)\in crit\,H := \{\boldsymbol{z}: \nabla_{\boldsymbol{z}}H(\boldsymbol{z}) = 0\}$.
	
	(ii) We have
	$$
	\lim_{k\rightarrow\infty}{dist}(\boldsymbol{z}^k,\omega(\boldsymbol{z}^0)) = 0,
	$$
	where ${dist}(\boldsymbol{z}^k,\omega(\boldsymbol{z}^0))$ denotes the distance from $\boldsymbol{z}^k$ to $\omega(\boldsymbol{z}^0)$.
\end{thm}
\begin{proof}
	(i) Let $\boldsymbol{z}^*=(x^*,\boldsymbol{y}^*)$ be the limit point of $\{\boldsymbol{z}^k\}_{k\in\mathbb{N}}=(x^k,\boldsymbol{y}^k)_{k\in\mathbb{N}}$.
	This means that there is a subsequence $(x^{k_q},\boldsymbol{y}^{k_q})_{q\in\mathbb{N}}$ such that $(x^{k_q},\boldsymbol{y}^{k_q})\rightarrow (x^*,\boldsymbol{y}^*)$ as $q\rightarrow\infty$.
	{Since $H$ is continuous (See Assumption \ref{assum:1}(ii), (iii))}, 
	we obtian that
	$$
	\lim_{q\rightarrow\infty} H(x^{k_q},\boldsymbol{y}^{k_q}) = H(x^*,\boldsymbol{y}^*).
	$$
	On the other hand we know from Lemma \ref{lem:convergence}(ii) and \ref{lem:gradient_gap} that $\nabla H(\boldsymbol{z}) \rightarrow \boldsymbol{0}$ as $k\rightarrow \infty$.
	This proves that $(x^*,\boldsymbol{y}^*)$ is a critical point of $H$.
	
	(ii) This item follows as an elementary consequence of the definition of limit points.
\end{proof}

\begin{remk}
	(i) {The simple structure of {the cubic NK-SVD algorithm} facilitates its extension to more general cases involving $R$ frequencies (i.e., {$2R$} blocks). In this scenario, {Theorem \ref{thm:critical}} remains applicable. The detailed derivations are omitted, as they closely parallel the developments presented for the single-frequency case.}
	
	(ii) We observe that the sparse coding stage can accelerate algorithm convergence. 
	By selecting useful atoms, the number of variables requiring iteration is reduced. 
	This reduction is beneficial because, in practical scenarios, the number of frequencies to be recovered is significantly smaller than the initial grid size $R$. 
	However, further theoretical analysis is needed and remains a focus for our future research.
	
	(iii) Our convergence results are not limited to the analysis of the proposed cubic NK-SVD algorithm but can also be applied to general alternating minimization algorithms with higher-order regularization \cite{Quartic,Cubic-quartic}. However, this is beyond the scope of this paper, and we consider further analysis as future work.
\end{remk}

\section{Simulation Results} \label{sec:simulation}

{In this section, we evaluate} the performance of the proposed {cubic NK-SVD} algorithm\footnote{Available at \url{https://github.com/xzliu-opt/Cubic-NK-SVD}} {in DOA estimation}.
{The experiments involve mixtures of $K$ sinusoids of length $N = 64$, with randomly generated frequencies $\{\theta_k,k=1,\cdots,K\}$.}
The true coefficients in $\boldsymbol{s}(t)$ are generated i.i.d. with uniform random phase on $[0,2\pi)$ and amplitudes drawn from a normal density of {mean 10 and variance 3}.
We use this specification to ensure that all components can be distinguished from noise.
The observation quality is measured by the peak-signal-to-noise ratio (PSNR) which is defined as $\text{PSNR} = 10\log_{10}(1/\sigma^2)$, where $\sigma^2$ denotes the noise variance.
We perform {300} Monte Carlo trials for each scenario.

We introduce {three} metrics to assess the recovery performance of the algorithms: the reconstruction signal-to-noise ratio (RSNR) \cite{sure-ir}, $\beta(\boldsymbol{\theta},\boldsymbol{\hat{\theta}})$ \cite{sbl-de}, and success rate. The RSNR is defined as:
$$
\operatorname{RSNR}=20 \log _{10}\left(\frac{\left\|\boldsymbol{A}(\boldsymbol{\theta}) \boldsymbol{s}(t)\right\|_2}{\left\|\boldsymbol{A}(\boldsymbol{\theta}) \boldsymbol{s}(t) - \boldsymbol{A}(\boldsymbol{\hat{\theta}}) \boldsymbol{\hat{s}}(t) \right\|_2}\right),
$$
which quantifies the accuracy of the reconstructed signal.
The second metric, $\beta(\boldsymbol{\theta}, \boldsymbol{\hat{\theta}})$, evaluates frequency reconstruction and is given by:
$$
\beta(\boldsymbol{\theta},\boldsymbol{\hat{\theta}}) = \frac{1}{K} \sum_{k=1}^{K} (\min_{\hat{\theta}\in\boldsymbol{{\hat{\theta}}}} d(\hat{\theta},\theta_k))^2,
$$
where $d(\hat{\theta},\theta_k)$ denotes the distance between the true and estimated frequencies.
Finally, the success rate represents the proportion of Monte Carlo trials in which the frequency vector $\boldsymbol{\theta}$ is successfully recovered.
A trial is deemed successful if the model order $K$ is correctly estimated and $\beta(\boldsymbol{\theta},\boldsymbol{\hat{\theta}}) < {10^{-3}}$.

\begin{figure}[t]
	\centering
	\includegraphics[width=2.3in]{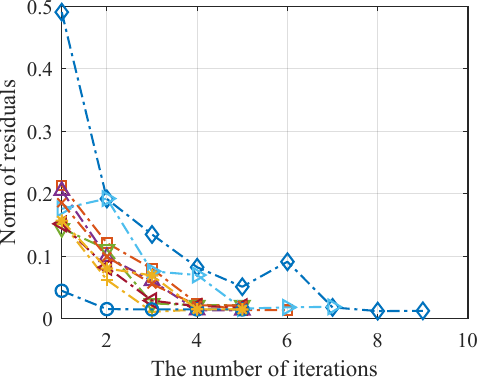}
	\caption{Norm of residuals versus the number of iterations in 10 Monte Carlo trials.}
	\label{fig:residual-iterpsnr20dbm32k3}
\end{figure}

We compare our proposed algorithm with other existing state-of-the-art ({SOTA}) super-resolution methods, namely, subspace method: MUSIC \cite{music};
On-Grid method: OMP \cite{omp}, $\ell_1$-SVD \cite{l1-svd};
Off-Grid methods: 
NOMP\footnote{Available at \url{https://bitbucket.org/wcslspectralestimation/continuous-frequency-estimation/src/NOMP/}} \cite{nomp},
MNOMP\footnote{Available at \url{https://github.com/RiverZhu/MNOMP-code}} \cite{mnomp};
Gridless methods: ANM\footnote{Available at \url{https://github.com/badrinarayan/astlinespec}} \cite{anm}, ANM-MMV\footnote{Available at \url{https://users.ece.cmu.edu/~yuejiec/papers/atomic_mmv.zip}} \cite{anm-mmv}, EMaC\footnote{Available at \url{https://users.ece.cmu.edu/~yuejiec/papers/EMaC_codes.zip}} \cite{emac}.
The $\text{Cram\'er}$-Rao bound (CRB) \cite{crb-stoica} serves as a benchmark, establishing the lower bound on the variance of any unbiased estimator.

For {the cubic NK-SVD algorithm}, we initialize the grid resolution with $R=N$, set the oversampling factor to $\gamma=10$, and define the error bound as $\epsilon=\sqrt{M} \sigma$.
The MUSIC algorithm \cite{music} estimates the autocorrelation matrix of the observed data vector, and in the MMV scenario with complete measurements, the total $T$ snapshots are used to generate this matrix.
Both  MUSIC and $\ell_1$-SVD require predefined initial grid points, which are set to 360 in our experiments to ensure recovery accuracy.
The true model order $K$ is supplied to these methods to pick out the top $K$ peaks using the MATLAB routine findpeaks.
For the OMP solution, a grid of size $2N=128$ is used, and it was observed that using a finer grid does not yield performance improvements.
The true model order $K$ is also used as the stopping criterion for the OMP algorithm.
For the NOMP and MNOMP algorithms, the oversampling factor $\gamma$ is set to 4 and the number of Newton steps $R_s$ is set to 1, with the number of cyclic refinements $R_c$ set to 3, as recommended in \cite{nomp,mnomp}.
The ANM, ANM-MMV and EMaC are solved using SDP.
The solutions to these gridless methods directly provide an estimate of the signal vector $\boldsymbol{A}(\boldsymbol{\theta})\boldsymbol{s}$ and the frequencies are recovered using the {root-MUSIC algorithm \cite{rootmusic}}, assuming the true model order $K$ is given.
The coefficients are then estimated via least-squares.

\begin{figure}[t]
	\centering
	\subfloat[]{\includegraphics[width=2.1in]{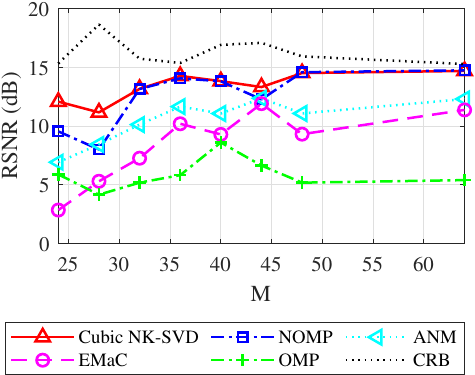}%
		\label{fig:rsnr_m_k=7_psnr=10}}
	\hfil
	\subfloat[]{\includegraphics[width=2.1in]{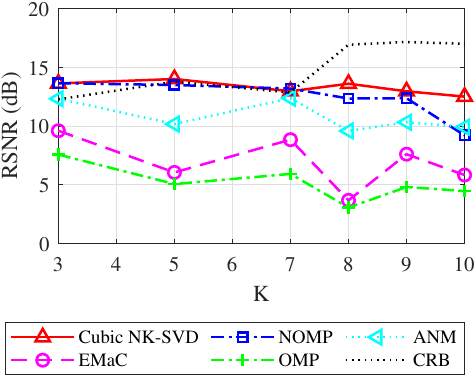}%
		\label{fig:rsnr_k_m=32_psnr=10}}
	\caption{{RSNRs} of respective algorithms. {(a) RSNRs vs. $M$, $K=7$ and PSNR=10 dB. (b) RSNRs vs. $K$, $M=32$ and PSNR=10 dB.}}
	\label{fig:rsnr_m_and_k_psnr=10}
\end{figure}

\begin{figure}[t]
	\centering
	\includegraphics[width=2.3in]{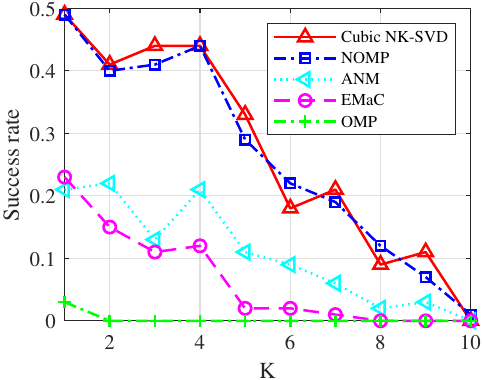}
	\caption{Success rates of respective algorithms vs. $K$, $M=32$ and PSNR=10 dB.}
	\label{fig:sr_k_m=32_psnr=10}
\end{figure}

\begin{figure}[t]
	\centering
	\subfloat[]{\includegraphics[width=1.7in]{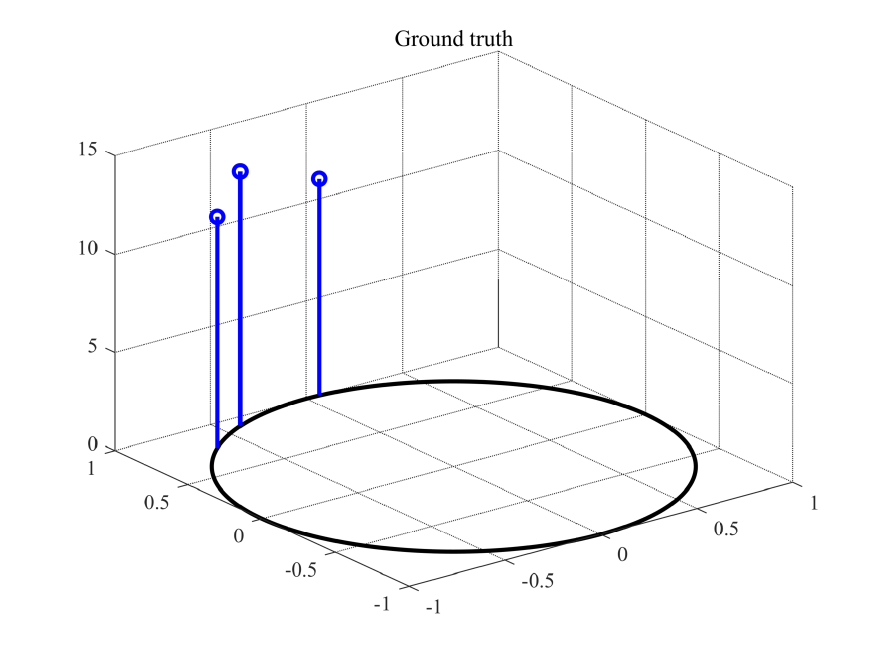}%
		\label{fig:plot3_m=24_k=3_psnr=0_truth}}
	\hfil
	\subfloat[]{\includegraphics[width=1.7in]{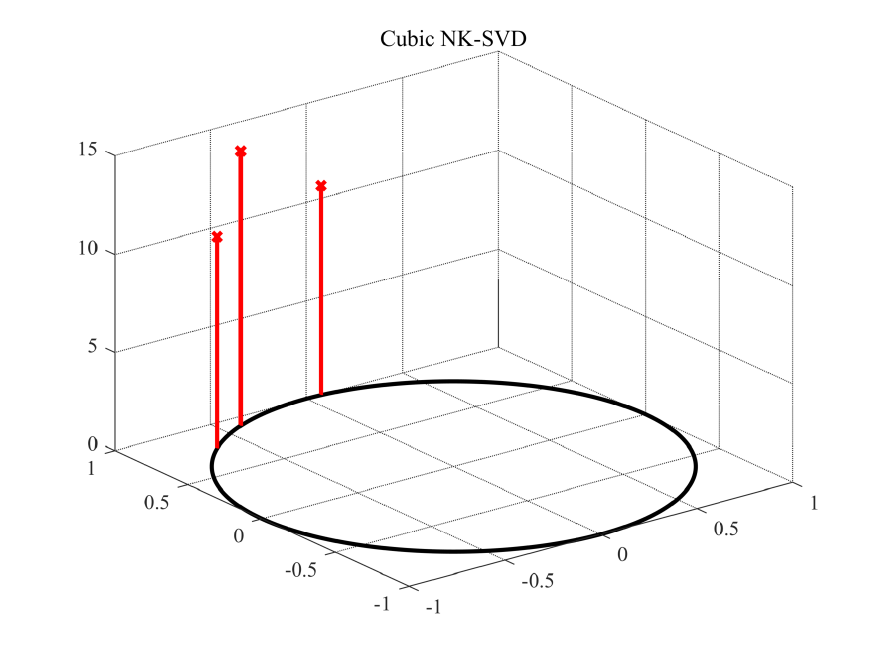}%
		\label{fig:plot3_m=24_k=3_psnr=0_nk_svd}}
	\hfil
	\subfloat[]{\includegraphics[width=1.7in]{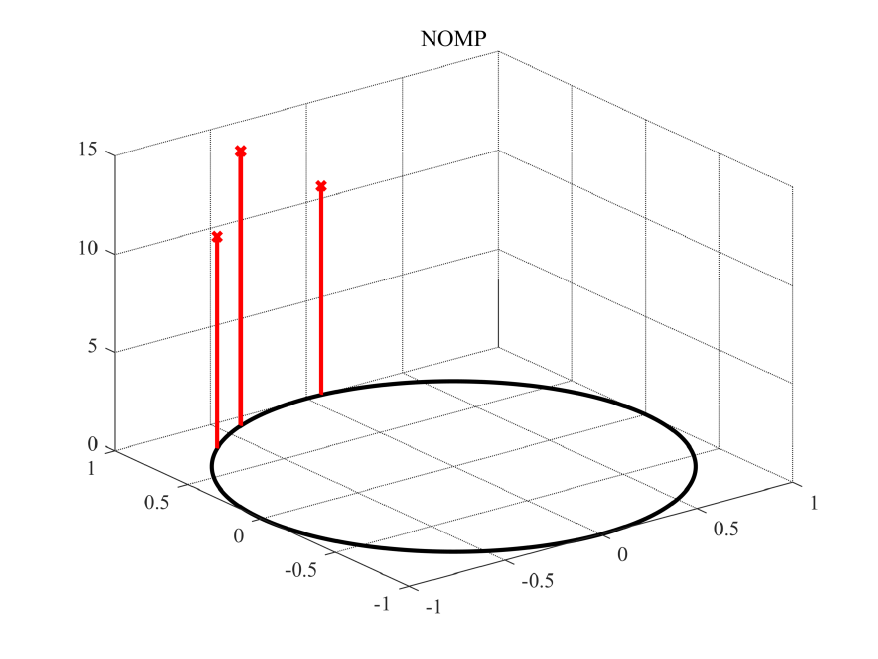}%
		\label{fig:plot3_m=24_k=3_psnr=0_nomp}}
	
	\subfloat[]{\includegraphics[width=1.7in]{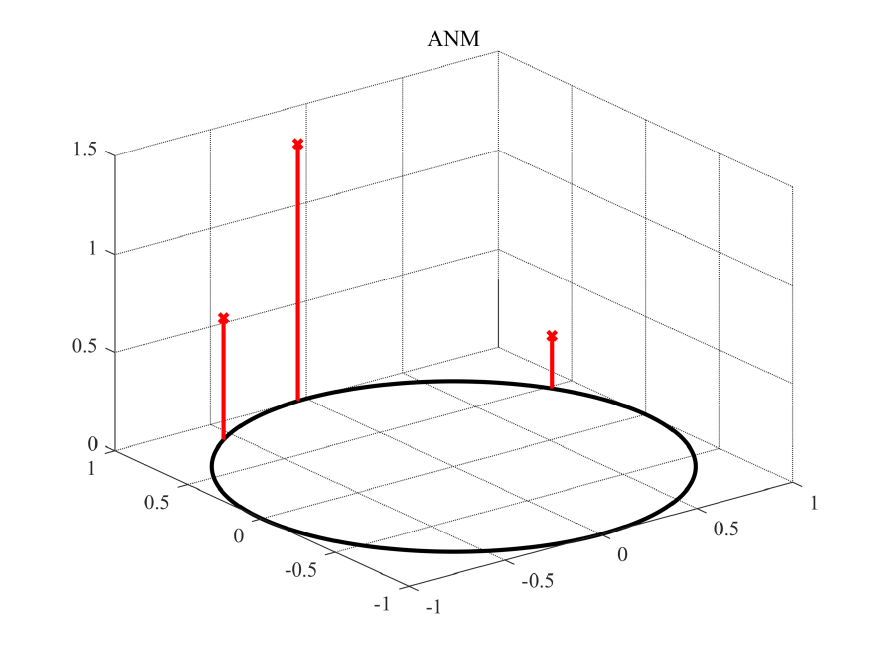}%
		\label{fig:plot3_m=24_k=3_psnr=0_anm}}
	\hfil
	\subfloat[]{\includegraphics[width=1.7in]{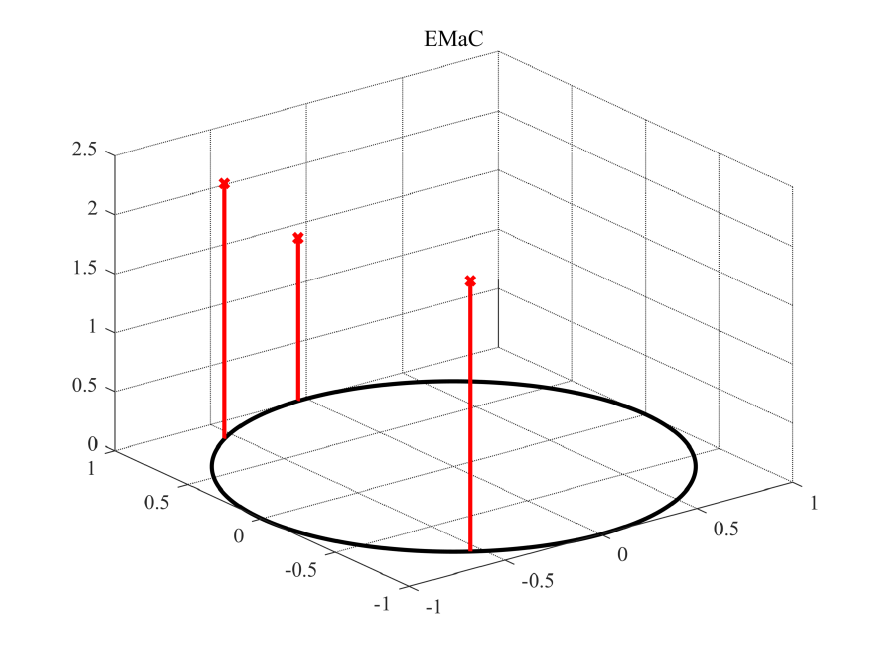}%
		\label{fig:plot3_m=24_k=3_psnr=0_emac}}
	\hfil
	\subfloat[]{\includegraphics[width=1.7in]{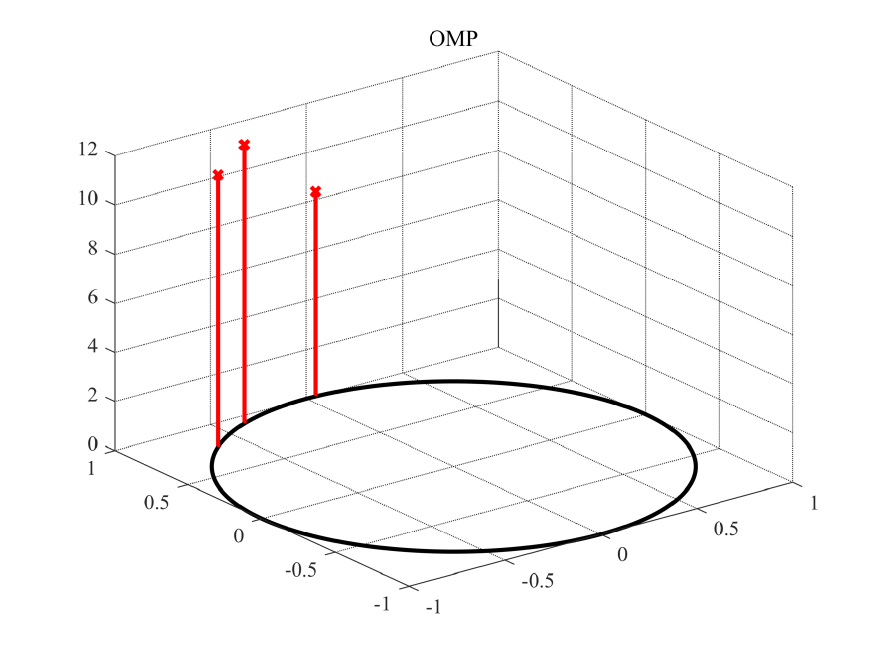}%
		\label{fig:plot3_m=24_k=3_psnr=0_omp}}
	
	\caption{Frequency estimation using different algorithms when $M=24$, $K=3$ and PSNR=0 dB. (a) Ground truth. (b) {Cubic NK-SVD}. (c) NOMP. (d) ANM. (e) EMaC. (f) OMP.}
	\label{fig:plot3_polar_m=24_k=3_psnr=0}
\end{figure}

\subsection{Frequency Estimation for SMV Model}

We investigate the performance of various algorithms in SMV scenarios, with $T=1$.
It is important to note that the effectiveness of many algorithms depends on minimum frequency separation, $\min_{k\neq l} |\theta_k - \theta_l|$,  between two sinusoids in the mixture, denoted as $\Delta \theta_{\text{min}}$.
According to \cite{anm-admm}, if $\Delta \theta_{\text{min}}$ exceeds {$4\pi/N$},
{the ANM approach} can exactly recover the frequencies in noise-free conditions.
We begin by consider an experimental setup where $\Delta \theta_{\text{min}}$ is guaranteed to be greater than {$4\pi/N$}. 
The sensing matrix $\boldsymbol{\Phi}$ is configured as a {random subsampling matrix},
meaning the measurements are obtained by randomly selecting $M$ entries from the $N$ elements of the signal.

We begin by evaluating the convergence behavior of the {cubic NK-SVD algorithm}.
In this experiment, we {set $M=32$, $K=3$, and PSNR=20 dB.}
The performance is assessed over 10 Monte Carlo trials.
Figure \ref{fig:residual-iterpsnr20dbm32k3} illustrates the {norm of residuals} in \eqref{eq:dl-problem} plotted against the number of iterations.
The results demonstrate that the residuals across all trials tend to converge, supporting the theoretical findings outlined in Theorem \ref{thm:critical}.

In Figure \ref{fig:rsnr_m_k=7_psnr=10}, we plot the average RSNRs of various algorithms as a function of the number of measurement, $M$, {with $K=7$ and PSNR=10 dB.}
The proposed method consistently demonstrates superior performance compared to the other four methods,
particularly excelling over the NOMP method when $M$ is small.
As $M$ increases, the recovery performance of the proposed method approaches the CRB.
In Figure \ref{fig:rsnr_k_m=32_psnr=10}, we present the average RSNRs vs. the number of frequencies, $K$, {with $M=32$ and PSNR=10 dB.}
It is evident that our method outperforms other algorithms, especially for larger values of $K$.
For instance, when $K=10$, our algorithm achieves an approximate 3 dB gain in RSNR compared to the NOMP and ANM methods.
Notably,for smaller values of $K$, the RSNR of the proposed algorithm exceeds the CRB because the cubic NK-SVD algorithm is biased.

In Figure \ref{fig:sr_k_m=32_psnr=10}, we present the {success rates} of various algorithms as a function of the the number of {frequencies}, $K$, where $M=32$ and PSNR is set to 10 dB.
The proposed method and the NOMP approach achieves significantly higher success rates compared to the gridless methods. This superiority is attributed to the gridless methods' sensitivity to noise, which leads to instability at lower PSNR levels. Additionally, the OMP method exhibits the poorest performance, further highlighting the critical need to address grid mismatch issues.

{Let $M=24$, $K=3$, and PSNR=0 dB.}
In Figure \ref{fig:plot3_m=24_k=3_psnr=0_truth}, we present a spectrally sparse ground truth scene.
{Figure \ref{fig:plot3_m=24_k=3_psnr=0_nk_svd}–\ref{fig:plot3_m=24_k=3_psnr=0_omp}} illustrate the estimated frequencies on a unit circle for various methods.
It is evident that the ANM and EMaC approaches exhibit poor frequency recovery in low PSNR scenarios. In contrast, the proposed method and NOMP approach demonstrate more accurate recovery. Notably, the performance of OMP is significantly compromised due to grid mismatch, preventing it from achieving super-resolution.

\begin{figure}[t]
	\centering
	\includegraphics[width=2.3in]{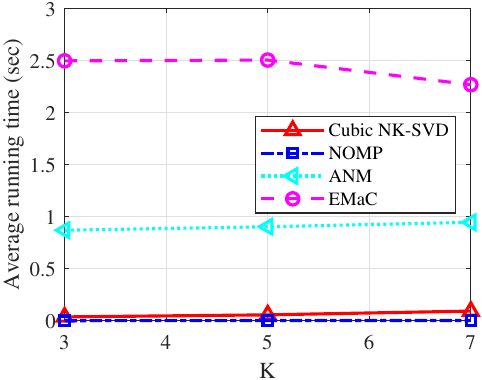}
	\caption{Average running times (sec) of respective algorithms when $M=24$ and PSNR=10 dB.}
	\label{fig:runtimekpsnr10dbm24}
\end{figure}

\begin{figure}[b]
	\centering
	\includegraphics[width=2.3in]{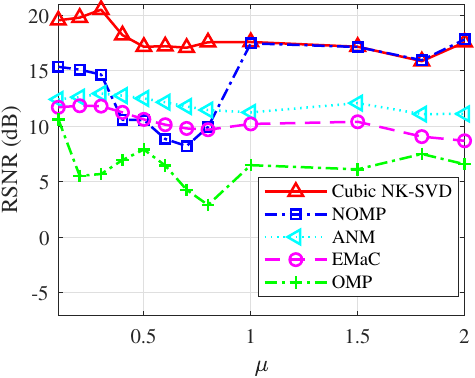}
	\caption{RSNRs of respective algorithms vs. $\mu$, $M=32$, $K=2$ and PSNR=10 dB.}
	\label{fig:rsnr-mupsnr10dbk2m32}
\end{figure}

Figure \ref{fig:runtimekpsnr10dbm24} presents the average running times of the various algorithms, with $M=24$ and PSNR=10 dB.
The results indicate that NOMP is the most computationally efficient algorithm.
The computational efficiency of cubic NK-SVD is notably higher than that of the gridless method \cite{anm,emac} and is comparable to NOMP. 
The performance of cubic NK-SVD is affected by the initial grid resolution $R$ and the oversampling factor $\gamma$. In practical applications, reducing the grid resolution can enhance computational efficiency without compromising accuracy.
Moreover, in MMV scenarios, integrating dimension reduction techniques \cite{xie_book} during the sparse coding stage can leverage the joint sparsity of the data while further reducing computational complexity. This strategy will be explored in our future work.

\begin{figure}[t]
	\centering
	\includegraphics[width=2.3in]{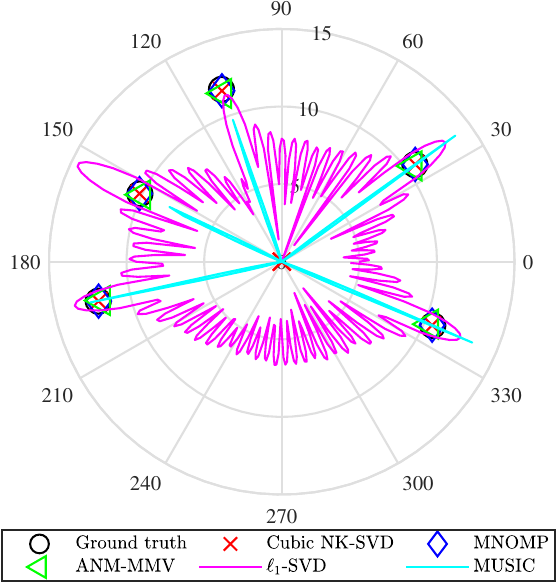}
	\caption{Angular spectra obtained using different algorithms when $M=64$, $K=5$ and PSNR=20 dB.}
	\label{fig:polark5m64t48psnr20}
\end{figure}

Finally, we evaluate the ability of various algorithms to recover closely-spaced frequencies.
We set $\Delta \theta_{\text{min}} = \mu \times 2\pi/N$, where $\mu$ is the frequency spacing coefficient.
Figure \ref{fig:rsnr-mupsnr10dbk2m32} illustrates the average RSNRs of different algorithms under varying $\mu$ values, with $M=32$, $K=2$, and PSNR set to 10 dB.
The results demonstrate that cubic NK-SVD significantly outperforms the other methods in this challenging scenario.
Notably, when $\mu<1$, the performance of NOMP becomes unstable and degrades significantly, whereas cubic NK-SVD maintains a substantial advantage. 
This performance discrepancy arises because NOMP, as a greedy algorithm, may approximate two closely spaced atoms with a single representative atom, particularly when $\mu$ is small. 
In contrast, cubic NK-SVD employs a predefined dictionary, enabling it to effectively exploit the interactions between atoms during iterations to achieve super-resolution.
Gridless methods struggle to accurately resolve frequency components in closely-spaced scenarios.
The OMP method exhibits the poorest performance in these scenarios, primarily due to grid mismatch. 
These experimental results highlight the superiority of cubic NK-SVD in handling the challenges associated with closely-spaced frequency recovery tasks.

\subsection{Frequency Estimation for MMV Model}

We evaluate the performance of {cubic NK-SVD} in the MMV scenarios.
In Figure \ref{fig:polark5m64t48psnr20}, the angular spectra of various algorithms are presented.
The performance of cubic NK-SVD is compared against MNOMP \cite{mnomp}, ANM-MMV \cite{anm-mmv}, $\ell_1$-SVD \cite{l1-svd}, and MUSIC \cite{music}.
As depicted, all methods accurately estimate the frequencies when
$M=64$, $\boldsymbol{\Phi}={\boldsymbol{I}}$, $T=48$, $K=5$, and PSNR=20 dB.
To further illustrate performance, we plot the $\beta(\boldsymbol{\theta},\boldsymbol{\hat{\theta}})$ of respective algorithms against PSNR in Figure \ref{fig:beta-psnrk5m64trial2}.
The results demonstrate that our approach significantly outperforms grid-based methods like $\ell_1$-SVD and MUSIC in achieving accurate frequency recovery, while maintaining performance comparable to ANM-MMV and MNOMP.
These findings highlight the advantages of addressing grid mismatch by operating directly in the continuous domain.

\begin{figure}[t]
	\centering
	\includegraphics[width=2.3in]{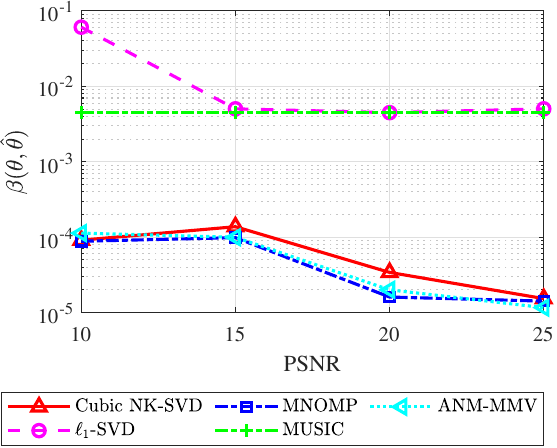}
	\caption{$\beta(\boldsymbol{\theta},\boldsymbol{\hat{\theta}})$ of respective algorithms vs. PSNR, $M=64$, $T=48$, and $K=5$.}
	\label{fig:beta-psnrk5m64trial2}
\end{figure}

\subsection{Real-World Dataset}

Our final experiment explores the recovery performance of Cubic NK-SVD using the widely adopted \texttt{NYUSIM} millimeter-wave (mmWave) channel dataset \cite{nyusim}. This dataset was developed to accurately model practical mobile systems operating at mmWave frequencies and serves as a benchmark for assessing estimation performance in wireless research. We consider a uniform linear array (ULA) channel signal in an urban microcell (UMi) non-line-of-sight (NLOS) scenario with a carrier frequency of 3.5 GHz.
The transmitter is equipped with $N_t=64$ antenna elements.
A random subsampling matrix is used as the sensing matrix $\boldsymbol{\Phi}$, and for each signal, $M=32$ samples are randomly selected. 
The PSNR is set to 15 dB.

Figure \ref{fig:nyusim} illustrates the true channel signal alongside the signals reconstructed by different algorithms.
As shown, the proposed Cubic NK-SVD algorithm achieves the best reconstruction performance, demonstrating a significant advantage over existing methods in this practical scenario.
Specifically, Cubic NK-SVD achieves an RSNR of 16.55 dB, outperforming NOMP (13.68 dB) and OMP (5.46 dB). These results highlight the superiority of our approach for real-world channel recovery applications.

\begin{figure}[t]
	\centering
	\subfloat[Ground truth]{\includegraphics[width=1.88in]{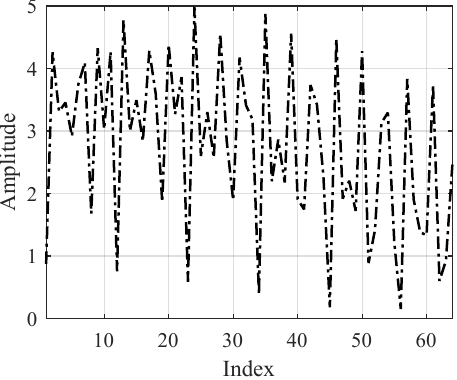}%
		\label{fig:nyusim_truth}}
	\hfil
	\subfloat[Cubic NK-SVD]{\includegraphics[width=1.88in]{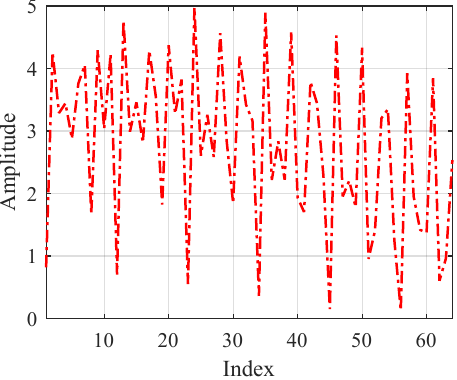}%
		\label{fig:nyusim_cubic}}
	
	\subfloat[NOMP]{\includegraphics[width=1.88in]{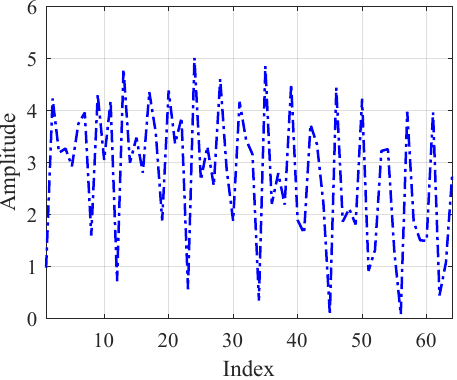}%
		\label{fig:nyusim_nomp}}
	\hfil
	\subfloat[OMP]{\includegraphics[width=1.88in]{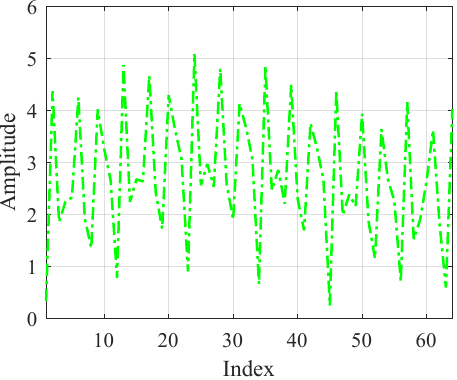}%
		\label{fig:nyusim_omp}}
	\caption{The true channel signal and the reconstructed signals obtained using different algorithms for $N_t=64$, $M=32$, and PSNR=15 dB.}
	\label{fig:nyusim}
\end{figure}

\section{Conclusion} \label{sec:conclusion}

In this paper, we addressed the LSE problem by leveraging the framework of parametric DL. We integrated the Cubic Newton refinement into the classic K-SVD algorithm, resulting in a novel parametric DL algorithm. 
This combination is natural and is thoroughly explained in our convergence analysis.
By building upon the BCD framework, we establish a rigorous convergence guarantee for our proposed cubic NK-SVD algorithm.
Our analysis not only solidifies the theoretical foundations of the proposed algorithm but also provides a flexible framework that can be extended to design other alternating minimization algorithms incorporating higher-order regularization. This opens up possibilities for applying our techniques to broader classes of BCD algorithms that utilize higher-order derivative information.
The proposed algorithm can be applied to both SMV and MMV scenarios, and an extension for its application in compressive scenarios is also provided.
Simulation results demonstrate that the proposed algorithm outperforms SOTA algorithms in various scenarios for DOA estimation.

\section*{Acknowledgement}
This work is supported by National Key Research and Development Program of China under grant 2021YFA1003303 and National Natural Science Foundation of China under grant 12171021.

\appendix
\section{{Convergence Analysis}}

\subsection{{Basic Properties}} \label{appendix:basic_property}

We begin by establishing fundamental properties of the cubic NK-SVD algorithm under Assumption \ref{assum:1}.

\begin{lem}[\cite{cubic_nesterov_06}]
	Let $h: \mathbb{R}^1 \rightarrow \mathbb{R}$ be a twice differentiable function with an $L_h$-Lipschitz continuous Hessian. Then, the following inequalities hold:
	$$
	\left|h^{\prime}(x_2)-h^{\prime}(x_1)-h^{\prime \prime}(x_1)(x_2-x_1)\right| \leq \frac{1}{2} L_h|x_2-x_1|^2,
	$$
	$$
	\left|h(x_2)-h(x_1)-\left\langle h^{\prime}(x_1), x_2-x_1\right\rangle -\frac{1}{2}\left\langle h^{\prime \prime}(x_1)(x_2-x_1), x_2-x_1\right\rangle\right| \leq \frac{L_h}{6}|x_2-x_1|^3.
	$$
\end{lem}

We define the modified Newton step using a cubic regularization of the quadratic approximation of the function as follows:
$$
T_{L_h}(x) \in \arg \min _{\tilde{x}} \left[ \left\langle h^{\prime}(x), {\tilde{x}}-x\right\rangle +\frac{1}{2}\left\langle h^{\prime \prime}(x)({\tilde{x}}-x), {\tilde{x}}-x\right\rangle +\frac{L_h}{6}|{\tilde{x}}-x|^3\right].
$$

The point $T_{L_h}(x)$ satisfies the following system of nonlinear equations:
$$
h^{\prime}(x)+h^{\prime \prime}(x)({\tilde{x}}-x)+\frac{1}{2} L_h|{\tilde{x}}-x| \cdot({\tilde{x}}-x)=0.
$$

Define
$$
\bar{h}_{L_h}(x) =  \min _{\tilde{x}} \left[ h(x) + \left\langle h^{\prime}(x), {\tilde{x}}-x\right\rangle + \frac{1}{2}\left\langle h^{\prime \prime}(x)({\tilde{x}}-x), {\tilde{x}}-x\right\rangle +\frac{L_h}{6}|{\tilde{x}}-x|^3\right].
$$

\begin{lem}[\cite{cubic_nesterov_06}] \label{lem:x_descent}
	For any $x \in \mathbb{R}^1$, the following holds:
	$$
	\bar{h}_{L_h}(x) \leq \min _{\tilde{x}} \left[h(\tilde{x})+\frac{L_h}{3}|\tilde{x}-x|^3\right],
	$$
	$$
	h(x)-\bar{h}_{L_h}(x) \geq \frac{L_h}{12} |T_{L_h}(x) - x|^3 .
	$$
	Moreover, we have
	$$
	h(T_{L_h}(x)) \leq \bar{h}_{L_h}(x).
	$$
\end{lem}

\begin{lem}[\cite{adaptation}] \label{lem:y_descent}
	Let $f: \mathbb{C}^T \rightarrow \mathbb{R}$ be a continuously differentiable, $\nu$-strongly convex function on $\mathbb{C}^T$, with $\nabla f(\boldsymbol{y}^*) = 0$.
	Then, 
	$$
	f(\boldsymbol{y}) \geq f(\boldsymbol{y}^*) + \frac{\nu}{2} \|\boldsymbol{y}-\boldsymbol{y}^*\|^2, \quad \forall \boldsymbol{y} \in \mathbb{C}^T.
	$$
\end{lem}

\begin{lem}[\cite{cubic_nesterov_06}] \label{lem:gradient_x}
	The point $T_{L_h}(x)$ satisfies
	\begin{equation}
		|h^{\prime}(T_{L_h}(x))| \leq L_h |T_{L_h}(x) - x|^2.
		\label{eq:x_gra_dec}
	\end{equation}
\end{lem}

\subsection{Proof of Lemma \ref{lem:convergence}} \label{appendix:proof_lem_convergence}

\begin{proof}
	(i) Fix $k\geq0$. Under our Assumption \ref{assum:1}, the functions $x \rightarrow H(x,\boldsymbol{y})$ (with $\boldsymbol{y}$ is fixed) and $\boldsymbol{y} \rightarrow H(x,\boldsymbol{y})$ (with $x$ is fixed) are twice differentiable function with Hessian assumed $L_h$-Lipschitz continuous and $\nu$-strongly convex function, respectively.
	Using the iterative steps \eqref{eq:x_subproblem} and \eqref{eq:y_subproblem}, applying Lemma \ref{lem:x_descent} and \ref{lem:y_descent}, first with $h(\cdot):=H\left(\cdot, \boldsymbol{y}^k\right)$ and $L_h = L(\boldsymbol{y}^k)$, and secondly with $f(\cdot):=H\left(x^{k+1}, \cdot\right)$, we obtain successively
	$$
	H(x^{k},\boldsymbol{y}^k) - H(x^{k+1},\boldsymbol{y}^k) \geq \frac{L(\boldsymbol{y}^k)}{12} |x^{k+1} - x^k|^3,
	$$
	and
	$$
	H(x^{k+1},\boldsymbol{y}^k) - H(x^{k+1},\boldsymbol{y}^{k+1}) \geq \frac{{\nu}}{2} \|\boldsymbol{y}^{k+1} - \boldsymbol{y}^k\|^2.
	$$
	Adding the above two inequalities, we thus obtain for all $k \geq 0$,
	\begin{equation}
		H(\boldsymbol{z}^{k}) - H(\boldsymbol{z}^{k+1}) \geq \frac{L(\boldsymbol{y}^k)}{12} |x^{k+1} - x^k|^3 + \frac{{\nu}}{2} \|\boldsymbol{y}^{k+1} - \boldsymbol{y}^k\|^2.
		\label{eq:obj_decrease}
	\end{equation}
	From \eqref{eq:obj_decrease} it follows that the sequence $\{H(\boldsymbol{z}^k)\}_{k\in\mathbb{N}}$ is nonincreasing, and since $H$ is assumed to be bounded below (see Assumption \ref{assum:1}(i)), {it converges to some real number $\underline{H}$}. Moreover, using the facts that $L(\boldsymbol{y}^k)\geq \lambda^- >0$ (see Assumption \ref{assum:1}(iv)), we get for all $k\geq 0$:
	\begin{equation}
		\begin{aligned}
			H(\boldsymbol{z}^{k}) - H(\boldsymbol{z}^{k+1}) \geq \frac{\lambda^-}{12} |x^{k+1} - x^k|^3+ \frac{{\nu}}{2} \|\boldsymbol{y}^{k+1} - \boldsymbol{y}^k\|^2, \\
		\end{aligned}
		\label{eq:obj_decrease_2}
	\end{equation}
	and assertion (i) is proved.
	
	(ii) Let $N$ be a positive integer. Summing \eqref{eq:obj_decrease_2} from $k=0$ to $N-1$ we also get
	$$
	\sum_{k=0}^{N-1} \frac{\lambda^-}{12} \left|x^{k+1}-x^k\right|^3 + \frac{{\nu}}{2} \left\|\boldsymbol{y}^{k+1}-\boldsymbol{y}^k\right\|^2 \leq \left(H\left(\boldsymbol{z}^0\right)-H\left(\boldsymbol{z}^N\right)\right) \leq \left(H\left(\boldsymbol{z}^0\right)-\underline{H}\right).
	$$
	Taking the limit as $N \rightarrow \infty$ we abtain
	\begin{equation}
		\sum_{k=0}^{\infty} \frac{\lambda^-}{12} \left|x^{k+1}-x^k\right|^3 + \frac{{\nu}}{2} \left\|\boldsymbol{y}^{k+1}-\boldsymbol{y}^k\right\|^2 \leq \left(H\left(\boldsymbol{z}^0\right)-H\left(\boldsymbol{z}^N\right)\right) \leq \left(H\left(\boldsymbol{z}^0\right)-\underline{H}\right) .
		\label{eq:obj_decrease_inf}
	\end{equation}
	and hence 
	$$
	\lim _{k \rightarrow \infty}\left|x^{k+1}-x^k\right|=0,
	$$
	$$
	\lim _{k \rightarrow \infty}\left\|\boldsymbol{y}^{k+1}-\boldsymbol{y}^k\right\|=0.
	$$
	Therefore, when $k$ is sufficiently large, both $\left|x^{k+1}-x^k\right|$ and $\left\|\boldsymbol{y}^{k+1}-\boldsymbol{y}^k\right\|$ are less than or equal to 1. Without loss of generality, we assume $\left|x^{k+1}-x^k\right| \leq 1$, $\left\|\boldsymbol{y}^{k+1}-\boldsymbol{y}^k\right\| \leq 1$ hold for all $k \geq 0$.
	We thus obtain
	\begin{equation}
		\left\|\boldsymbol{y}^{k+1}-\boldsymbol{y}^k\right\|_3^3 \leq \left\|\boldsymbol{y}^{k+1}-\boldsymbol{y}^k\right\|^2, \quad \forall k \geq 0.
		\label{eq:y_cubic}
	\end{equation}
	Combining \eqref{eq:obj_decrease_inf} and \eqref{eq:y_cubic} yields the following
	$$
	\sum_{k=0}^{\infty} \rho_1 \left|x^{k+1}-x^k\right|^3 + \rho_1 \left\|\boldsymbol{y}^{k+1}-\boldsymbol{y}^k\right\|_3^3
	=\sum_{k=0}^{\infty} \rho_1 \left\|\boldsymbol{z}^{k+1}-\boldsymbol{z}^k\right\|_3^{3} \leq \left(H\left(\boldsymbol{z}^0\right)-\underline{H}\right),
	$$
	where $\rho_1 = \min \{\frac{\lambda^-}{12}, \frac{\nu}{2}\}$. We obtain the desired assertion (ii).
\end{proof}

\subsection{Proof of Lemma \ref{lem:gradient_gap}} \label{appendix:proof_gradient_gap}

\begin{proof}
	Let $k$ be a positive integer. 
	In view of inequality \eqref{eq:x_gra_dec}, we have
	$$
	\left\|\nabla_x H(x^{k+1},\boldsymbol{y}^{k}) \right\| \leq L(\boldsymbol{y}^k) \left|x^{k+1} - x^{k}\right|^2.
	$$
	Writing down the optimality condition for \eqref{eq:y_subproblem} yields
	\begin{equation}
		\|\nabla_{\boldsymbol{y}} H(x^{k+1},\boldsymbol{y}^{k+1}) \| = 0.
		\label{eq:lem_grad_y}
	\end{equation}
	Note that $H(x,\boldsymbol{y})$ is defined in \eqref{eq:H_doa}, we have
	\begin{equation}
		\begin{aligned}
			& \|\nabla_x H(x^{k+1},\boldsymbol{y}^{k+1})  - \nabla_x H(x^{k+1},\boldsymbol{y}^{k}) \| \\
			& = \left|2 \Re \left\{(\boldsymbol{y}^{k+1} - \boldsymbol{y}^k)^H E^H \frac{\partial  \boldsymbol{a}(x^{k+1})}{\partial x}\right\} \right| \\
			& = \left|2 \Re  \left\{(\boldsymbol{y}^{k+1} - \boldsymbol{y}^k)^H \boldsymbol{y}^{k+1}\right\}  \right|.
		\end{aligned}
		\label{eq:lem_grad_z}
	\end{equation}
	Then,
	$$
	\left\|\nabla_x H(x^{k+1},\boldsymbol{y}^{k+1}) \right\| \leq \left\|\nabla_x H(x^{k+1},\boldsymbol{y}^{k}) \right\| + 2\left\| (\boldsymbol{y}^{k+1} - \boldsymbol{y}^k)\right\| \left\| \boldsymbol{y}^{k+1} \right\|.
	$$
	Since we assumed that $\{\boldsymbol{z}^k\}_{k\in\mathbb{N}}$ is bounded, we get that
	$$
	\left\|\boldsymbol{y}^k\right\| \leq \left\|\boldsymbol{z}^k\right\| \leq \sigma,\quad \forall k \in \mathbb{N}.
	$$
	The moduli  $L(\boldsymbol{y}^k)$ being bounded from above by $\lambda^+$ (see Assumption \ref{assum:1}(iv)).
	In view of inequality \eqref{eq:lem_point_conver}, we get that when $k$ is sufficiently large, $\left|x^{k+1}-x^k\right|$ is less than or equal to 1. Without loss of generality, we assume $\left|x^{k+1}-x^k\right| \leq 1, \forall k \geq 0$.
	Summing up these inequalities, we get
	\begin{equation}
		\begin{aligned}
			\left\|\nabla_x H(x^{k+1},\boldsymbol{y}^{k+1}) \right\| \leq &  \lambda^+ \left\|x^{k+1} - x^{k}\right\|^2 + 2\sigma\left\| (\boldsymbol{y}^{k+1} - \boldsymbol{y}^k)\right\| \\
			\leq &  \lambda^+ \left\|x^{k+1} - x^{k}\right\| + 2\sigma\left\| (\boldsymbol{y}^{k+1} - \boldsymbol{y}^k)\right\| \\
			\leq &  \rho_2 \left\|x^{k+1} - x^{k}\right\| + \rho_2 \left\| (\boldsymbol{y}^{k+1} - \boldsymbol{y}^k)\right\| \\
			\leq &  \rho_2 \left\|\boldsymbol{z}^{k+1} - \boldsymbol{z}^{k}\right\|.
		\end{aligned}
		\label{eq:lem_grad_x_x}
	\end{equation}
	which is the desired result in \eqref{eq:lem_grad_x}.
	Then, in view of inequality \eqref{eq:lem_grad_y} and \eqref{eq:lem_grad_x_x}, we get the estimate \eqref{eq:lem_grad_z}.
	This completes the proof.
\end{proof}

\section{Properties of the Function in \eqref{eq:H_doa}} \label{appendix:property_H}

In this section, we demonstrate that the function $H(x,{\boldsymbol{y}})$ in \eqref{eq:H_doa} satisfies Assumption \ref{assum:1}.
It is evident that $H(x,{\boldsymbol{y}})\geq 0$.

Notice that
$$
\begin{aligned}
	&\nabla_xH(x,{\boldsymbol{y}}) = -2\Re\{{\boldsymbol{y}^H}\boldsymbol{E}^H{\frac{d \boldsymbol{a}}{d x}}\},\\
	&\nabla^2_{xx}H(x,{\boldsymbol{y}}) = -2\Re\{{\boldsymbol{y}^H}\boldsymbol{E}^H{\frac{d^2 \boldsymbol{a}}{d x^2}}\} = -2\Re\{{\boldsymbol{y}^H}\boldsymbol{E}^H{\boldsymbol{D}^2}\boldsymbol{a}(x)\}.
\end{aligned}
$$
where ${\frac{d \boldsymbol{a}}{d x}} = \boldsymbol{D}\boldsymbol{a}(x)$ and $\boldsymbol{D} = {\operatorname{Diag}}(0,j,2j,\cdots,(N-1)j)$.
Thus,
\begin{equation}
	\begin{aligned}
		|\nabla^2_{xx}H(x_1,{\boldsymbol{y}}) - \nabla^2_{xx}H(x_2,{\boldsymbol{y}})| &= 2\Re\{{\boldsymbol{y}^H}\boldsymbol{E}^H{\boldsymbol{D}^2}(\boldsymbol{a}(x_1) - \boldsymbol{a}(x_2))\} \\
		&\leq 2\|{\boldsymbol{D}^2}\boldsymbol{E}\boldsymbol{{y}}\|\|\boldsymbol{a}(x_1)-\boldsymbol{a}(x_2)\|.
	\end{aligned}
	\label{eq:hessian_bound}
\end{equation}
Let $g(x) = \frac{1}{jk}e^{jkx}$.
Since $|\frac{d^2}{dx^2}g(x)| \leq k$, $g(x)$ is {Lipschitz continuous} with moduli $L=k$. Therefore,
$$
|\frac{d}{dx}g(x_1) - \frac{d}{dx}g(x_2)|=|e^{jkx_1}-e^{jkx_2}| \leq k |x_1-x_2|, \forall x_1,x_2\in \mathbb{R}.
$$
Then,
\begin{equation}
	\begin{aligned}
		\|\boldsymbol{a}(x_1)-\boldsymbol{a}(x_2)\|^2 &= \frac{1}{N} \sum_{k=1}^{N-1} (e^{jkx_1}-e^{jkx_2})^2 \leq \frac{1}{N} \sum_{k=1}^{N-1} k^2 (x_1 - x_2)^2\\
		&=\frac{(N-1)(2N-3)}{6} (x_1-x_2)^2.
	\end{aligned}
	\label{eq:a_bound}
\end{equation}
From \eqref{eq:hessian_bound} and \eqref{eq:a_bound}, we obtain
$$
|\nabla^2_{xx}H(x_1,{\boldsymbol{y}}) - \nabla^2_{xx}H(x_2,{\boldsymbol{y}})| \leq L(\boldsymbol{{y}}) |x_1-x_2|, \forall x_1,x_2\in \mathbb{R},
$$
where $L(\boldsymbol{{y}}) = \sqrt{\frac{2(N-1)(2N-3)}{3}}\|{\boldsymbol{D}^2}\boldsymbol{E}\boldsymbol{{y}}\|$.
This indicates that the function $x \rightarrow H(x,\boldsymbol{y})$ is $C_{L(\boldsymbol{y})}^{2,2}$.

For any fixed $x$, we have
$$
{(\nabla^2_{\boldsymbol{y}\boldsymbol{y}})_c H(x,\boldsymbol{y})} = {\boldsymbol{I}_{2T}} \succ {\boldsymbol{0}},
$$
which indicates that the function $\boldsymbol{y} \rightarrow H(x,\boldsymbol{y})$ is $\nu$-strongly convex with $\nu=2$.

According to \cite{bcd_bolte_14},
consider a function $H(x,\boldsymbol{y})$ whose partial Hessian $\nabla^2_{xx}H(x,\boldsymbol{y})$ is Lipschitz continuous.
Let $\mu^-$ be an arbitrary positive constant, and replace the Lipschitz modulus $L(\boldsymbol{y})$ by $L^{\prime}(\boldsymbol{y}) = \max\{L(\boldsymbol{y}),\mu\}$.
Then, $L^{\prime}(\boldsymbol{y})$ remains a Lipschitz modulus for $\nabla^2_{xx}H(x,\boldsymbol{y})$.
Consequently, the inequality in \eqref{eq:lip_lower_bound} is trivially satisfied with $L^{\prime}(\boldsymbol{y})$ and $\lambda^{-} = \mu^{-}$.
Under the assumption that the sequence $\{\boldsymbol{z}^k\}_{k\in\mathbb{N}}$ generated from the proposed algorithm is bounded, the inequality in \eqref{eq:lip_upper_bound} holds because $L(\boldsymbol{{y}})$ is upper-bounded by $\|\boldsymbol{{y}}\|$ with constant scaling.

\bibliographystyle{elsarticle-num} 
\bibliography{mybib}

\end{document}